\documentclass[a4paper,12pt,reqno,oneside]{amsart} % amsart loads amsmath, amsthm and amsfonts

\usepackage{setspace} 
\usepackage{typearea} % European margins
\usepackage{amscd,amssymb,verbatim} % amssymb loads amsfonts
\usepackage%[pdftex]
{graphicx} \setkeys{Gin}{width=\linewidth}
\usepackage{enumitem}
\usepackage{mathdots}

\usepackage{xr-hyper}
\usepackage{cite}
\usepackage[colorlinks,final,bookmarksnumbered,bookmarks]{hyperref}
\usepackage{amsrefs}

\usepackage{ifthen}
\newcommand{\arxiv}[2][]{\ifthenelse{\equal{#1}{}}
{\href{http://arxiv.org/abs/#2}{\tt arXiv:#2}}
{\href{http://arxiv.org/abs/math/#2}{\tt arXiv:math.#1/#2}}}

\usepackage[utf8]{inputenc}
\usepackage[T1,T2A]{fontenc}
\usepackage[russian]{babel}

\usepackage{ulem}

\theoremstyle{plain}
\newtheorem{theorem}{Теорема}[section]
\newtheorem{lemma}[theorem]{Лемма}
\newtheorem{corollary}[theorem]{Следствие}

\theoremstyle{definition}
\newtheorem{definition}[theorem]{Определение}

\newtheoremstyle{remark}
{}{}{}{}{\itshape}{}{ }{\thmname{#1}\thmnumber{ \itshape #2.}}
\theoremstyle{remark}
\newtheorem{remark}[theorem]{Замечание}
\newtheorem{example}[theorem]{Пример}

\def\phi{\varphi}

\def\R{\Bbb{R}}

\makeatletter
\def\@settitle{\begin{center}%
    \baselineskip14\p@\relax
    \bfseries
    \@title
  \end{center}%
}
\makeatother

\begin{document}

%\tableofcontents

\title[]{О ЧИСЛЕ ГРАНЕЙ МНОГОГРАННИКОВ ГЕЛЬФАНДА--ЦЕТЛИНА}
\author{Е. В. Мелихова}

\begin{abstract}
Работа посвящена изучению комбинаторики многогранников Гель\-фанда--Цетлина. С использованием  геометрических свойств линейной проекции многогранника Гельфанда--Цетлина на некоторый куб выведено рекуррентное соотношение на $f$-многочлен многогранника Гельфанда--Цетлина. Для многогранников Гельфанда--Цетлина нескольких простейших типов, образующих однопараметрические семейства, с помощью найденного рекуррентного соотношения явно выписаны $f$-многочлен и $h$-многочлен.
 \end{abstract}
 
\address{Ekaterina V. Melikhova,
National Research University ``Higher School of Economics'',
Faculty of Mathematics,
Usacheva St., 6, Moscow, 119048, Russia}
\email{ekmelikhova86@gmail.com}

\maketitle

\section{Введение}
Пусть $\lambda_1 \le \dots \le \lambda_s$ --- неубывающая конечная последовательность действительных чисел. Рассмотрим треугольную таблицу 

\begin{equation}\label{tab}
\setcounter{MaxMatrixCols}{11}
\begin{matrix}
	\lambda_1&     &     \lambda_2&     &     \lambda_3&    &    \lambda_4&   &   \dots&     &      \lambda_s\\
	   &         u_{1,1}&    &        u_{1,2}&    &      u_{1,3}&   &       \dots&  &     u_{1,s-1}\\
	   &           &     u_{2,1}&       &      u_{2,2}&     &     \dots&       &   u_{2,s-2}\\
	   &           &         &        \ddots&  \ddots&   \vdots&  \iddots&       \iddots\\
	   &           &         &          &      u_{s-2,1}&   &   u_{s-2,2}\\
	   &           &         &          &         &      u_{s-1,1}\\
\end{matrix}.
\end{equation}

 \medskip
 
Будем считать, что каждые три символа $a,\ b,\ c$, которые в таблице стоят в вершинах треугольника
\[
\begin{matrix}
a& & c\\
&  b
\end{matrix},
\]
 связаны двойным неравенством: $a\le b \le c$. Тогда наша таблица даёт конкретный набор линейных неравенств, зависящих от $\lambda_i$, который определяет некоторый выпуклый многогранник в $\R^{\frac{s(s-1)}{2}}$. Заметим, что координаты $u_{i,j}$ в нашем случае занумерованы упорядоченной парой целых чисел $(i,\,j)$, где $i$ пробегает значения от $1$ до $s-1$, а $j$ пробегает значения от $1$ до $s-i$.
Этот выпуклый многогранник и называется \textit{многогранником Гельфанда--Цетлина, соответствующим последовательности}
 $\lambda_1 \le \dots \le \lambda_s$. Обозначим его через $GZ(\lambda_1 \dots \lambda_s)$. 
 \begin{remark}
    Если в последовательности $\lambda_1 \le \dots \le \lambda_s$ $k$ различных чисел с кратностями $i_1,\, i_2,\, \dots i_k$ соответственно, то $\dim GZ(\lambda_1 \dots \lambda_s)=\frac{1}{2}(s^2-\sum_{j=1}^{k}{i_j}^2)$.
 \end{remark}  
\begin{example}\label{simplexGZ}
  Рассмотрим многогранник $GZ(0\, 1^l)$. (Здесь и далее мы используем мультипликативную запись для последовательности $\lambda_1 \le \dots \le \lambda_s $ с повторениями: показатель равен количеству копий основания. В частности, $GZ(0 \,1^l)$ -- это многогранник Гельфанда--Цетлина, отвечающий последовательности из одного нуля и $l$ единиц.) Соответствующая треугольная таблица равносильна цепочке нестрогих неравенств $0 \le u_{1,1} \le u_{2,1} \le \dots \le u_{l,1} \le 1$ и, следовательно,  $GZ(0\, 1^l)$ является $l$-мерным симплексом. (Для многогранника $GZ(0^l\, 1)$ соответствующая треугольная таблица равносильна цепочке $0 \le u_{l,1} \le u_{l-1,2} \le \dots \le u_{1,l} \le 1$ и, следовательно,  $GZ(0^l\, 1)$ также является $l$-мерным симплексом.) 
\end{example}   
Многогранники Гельфанда--Цетлина берут своё начало в теории представлений группы $GL(n)$
(см.\ \cite{Sm16}*{\S5.2} и \cite{GZ}) и тесно связаны с исчислением Шуберта и топологией
многообразий полных флагов (см.\ \cite{Sm16} и \cite{KST}).
Так, кольцо когомологий многообразия полных флагов $Fl(n)$ изоморфно кольцу
Хованского--Пухликова многогранника Гельфанда--Цетлина $GZ(\lambda_1\dots\lambda_n)$,
где $\lambda_1<\dots<\lambda_n$ (К. Каве; см.\ \cite{Sm16}*{\S5.1}), причём умножение
в этом кольце описывается в терминах пересечения граней многогранника $GZ(\lambda_1\dots\lambda_n)$ (см.\ \cite{KST} и \cite{Sm16}*{\S5.4}).
Таким образом, комбинаторика многогранников Гельфанда--Цетлина представляет определённый интерес.

П. Гусев, В. Кириченко и В. Тиморин нашли, кроме прочего, рекуррентное соотношение на числа вершин многогранников $GZ(\lambda_1\dots\lambda_n)$,
где $\lambda_1\le\dots\le\lambda_n$ \cite{GKT}.
 
В настоящей работе получено, в частности, рекуррентное соотношение на $f$-векторы этих многогранников (определение $f$- вектора см. ниже).

Напомним, что \textit{решёткой граней} $L(P)$ многогранника $P$ называют множество его граней, упорядоченное по включению, и что два многогранника $P$ и $P^{\prime}$ называют \textit{комбинаторно эквивалентными}, если их решётки граней изоморфны, то есть существует биекция $L(P)$ на $L(P^{\prime})$, сохраняющая отношение включения. Если многогранник $P$ комбинаторно эквивалентен многограннику $P^{\prime}$, мы будем писать $P\simeq P^{\prime}$. Одной из важнейших характеристик класса комбинаторно эквивалентных многогранников является конечная последовательность $(f_0, f_1,\dots, f_n)$, где $n$ -- размерность многогранника, $f_i$ -- число граней размерности $i$. Её называют $f$-\textit{вектором}. Соответствующую производящую  функцию $f(t)=f_0+f_1\cdot t+\dots+ f_n \cdot t^n$ называют $f$-\textit{многочленом}. В некоторых случаях нам будет удобнее рассуждать на языке $h$-\textit{многочлена}. Последний восстанавливается по $f$-многочлену однозначно, а именно: $h(s):=f(s-1)=\sum_{i=0}^{n}h_is^i$, $h_i=\sum_{k= i}^{n} f_k (-1)^{k-i}\binom ki $. Конечную последовательность $(h_0, h_1,\dots, h_n)$ называют $h$-\textit{вектором}. 
 
Легко видеть, что (с точностью до комбинаторной эквивалентности) существует всего два трёхмерных многогранника Гельфанда--Цетлина: $GZ(1\, 2\, 3)$ и симплекс $GZ(0\, 1^3)$. Поэтому, если нам понадобится геометрическая иллюстрация, мы будем обращаться к многограннику $GZ(1\, 2\, 3)$, см. рисунок \ref{howtogetGZ(123)}. 
 \begin{figure}[h]
  \begin{minipage}[h]{0.49\linewidth}
   \[\begin{matrix}
   1 && 2 && 3\\
   & u_{1,1} && u_{1,2}\\
   && u_{2,1}\\
   \end{matrix}
    \Leftrightarrow
    \begin{cases}
   1\le u_{1,1}\le 2,\\
   2\le u_{1,2}\le 3,\\
   u_{1,1}\le u_{2,1}\le u_{1,2};
   \end{cases}
   \]	
  \end{minipage}
 \hfill
 \hfill
 \begin{minipage}[h]{0.49\linewidth}
   \center{\includegraphics[width=0.9\linewidth]{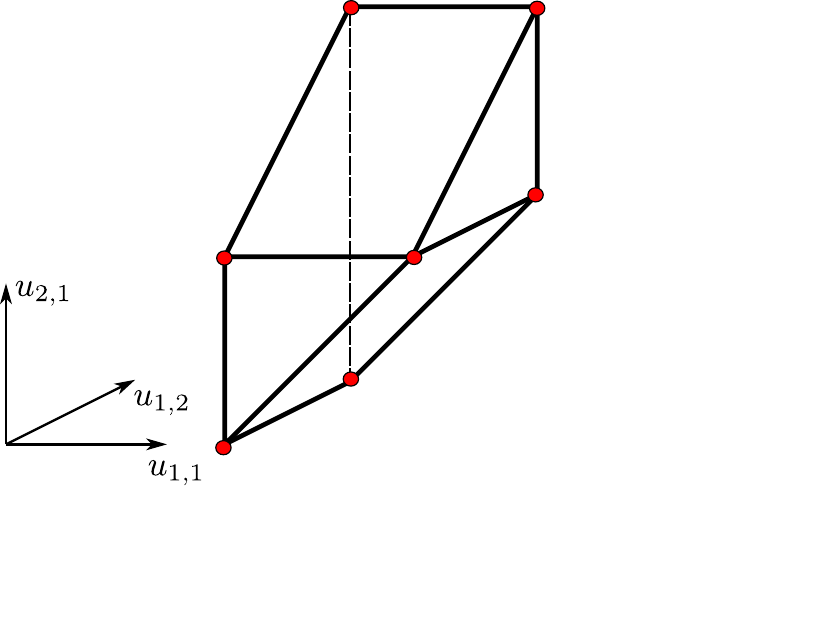}}
 \end{minipage}
 \caption{Многогранник $GZ(1\, 2\, 3)$, определяющая его треугольная таблица и соответствующая ей система неравенств. }
   \label{howtogetGZ(123)}
 \end{figure}
 
Скажем более подробно о полученных результатах. В разделе \ref{2} мы выводим упомянутое выше рекуррентное соотношение, пользуясь геометрическими свойствами линейной проекции на некоторый куб, впервые построенной в работе \cite{GKT}. Ключевой момент состоит в том, что эта проекция настолько хороша, что каждая грань многогранника Гельфанда--Цетлина отображается на некоторую грань куба, и для каждой грани куба полный прообраз любой её внутренней точки комбинаторно эквивалентен некоторому многограннику Гельфанда--Цетлина.  Полученное рекуррентное соотношение \ref{recurrence relation} выражает $f$-многочлен исходного многогранника Гельфанда--Цетлина через $f$-многочлены многогранников, комбинаторно эквивалентных полным прообразам барицентров граней куба при рассматриваемой проекции.   Геометрическая часть доказательства теоремы \ref{recurrence relation} заключена в леммах \ref{lem1}, \ref{main lemma} и следствии \ref{main corollary}. 
 
В разделе \ref{Calculate} мы применяем найденное рекуррентное соотношение для исследования многогранников Гельфанда--Цетлина, соответствующих неубывающим последовательностям $(1 2^k 3)$, $(1 2 3^k)$ и $(2 2 3^k)$.
Пункт \ref{subsecGZ123k}  посвящён многогранникам $GZ(12^k3)$. Он доказывает, что, говоря неформально, $h_{GZ(12^k3)}$ имеет вид несимметричной горки: его компоненты сначала возрастают, а затем убывают, причём если исключить из рассмотрения компоненту старшей размерности, то любые две соседние компоненты отличаются на единицу, см. теорему \ref{hvector12k3}. В пункте \ref{subsecGZ123k223k} мы изучаем многогранники $GZ(123^k)$ и $GZ(223^k)$. Результатом является теорема \ref{hpolynom123k223k}, которая даёт достаточно компактные формулы для $h$-многочленов многогранников $GZ(123^k)$ и $GZ(223^k)$; а также следствие \ref{gfh123k223k}, которое по этим формулам находит производящие функции последовательностей $\{h_{GZ(123^k)}\}$ и $\{h_{GZ(223^k)}\}$.
\begin{remark}
Рекуррентное соотношение на $f$-многочлен многогранников Гель\-фанда--Цетлина получено также в недавней работе Ан, Чо и Кима \cite{BYJ}, о чём я узнала уже после того, как записала набросок доказательства теоремы \ref{recurrence relation}. 

Подход авторов работы \cite{BYJ} заметно отличается от используемого в данной статье. Авторы работы \cite{BYJ} дают описание решётки граней многогранника Гельфанда--Цетлина в терминах некоторого ориентированного плоского графа, называемого \textit{ступенчатой диаграммой} (\textit{ladder diagram}). Каждому классу комбинаторно эквивалентных многогранников Гельфанда--Цетлина они ставят в соответствие некоторую ступенчатую диаграмму и доказывают, что их решётки граней изоморфны. (<<Гранями>> ступенчатой диаграммы авторы называют некоторые её подграфы специального вида.) Пользуясь только комбинаторными свойствами ступенчатых диаграмм, авторы выводят рекуррентное соотношение, которое описывает $f$-многочлен ступенчатой диаграммы через $f$-многочлены ступенчатых диаграмм меньших размерностей. 
\end{remark}
\section{Рекуррентное соотношение}\label{2}
Мы будем работать с классами комбинаторно эквивалентных многогранников Гельфанда--Цетлина, для удобства в каждом классе выберем для себя эталонного представителя. 
\begin{remark}\label{AlmostOb} Достаточно очевидно, что если между наборами чисел $\lambda_1 \le \dots \le \lambda_s$ и $\lambda_1^{\prime} \le \dots \le \lambda_s^{\prime}$ существует биекция, сохраняющая отношения $=$ и $<$, тогда соответствующие многогранники Гельфанда--Цетлина комбинаторно эквивалентны. В частности,  если $\lambda_1 < \lambda_2 < \dots < \lambda_k$ , то  $GZ(\lambda_1^{i_1} \lambda_2^{i_2} \dots \lambda_k^{i_k} )\sim GZ( 1^{i_1}2^{i_2}\dots k^{i_k})$. 
\end{remark}

\begin{remark}
Менее очевидно, что если между наборами чисел $\lambda_1 \le \dots \le \lambda_s$ и $\lambda_1^{\prime} \le \dots \le \lambda_s^{\prime}$ существует биекция, сохраняющая отношение $=$, но обращающая строгий порядок, то соответствующие многогранники Гельфанда--Цетлина также будут комбинаторно эквивалентны. См. лемму \ref{onelemma}.
\end{remark}

В дальнейших рассуждениях нам будет удобно в качестве эталонного представителя своего комбинаторного класса использовать многогранник  $GZ(1^{i_1}\,2^{i_2}\,\dots \,k^{i_k})$. Cледуя статье \cite{GKT} , рассмотрим линейную проекцию $\pi$ многогранника $GZ ( 1^{i_1}\,2^{i_2}\,\dots \,k^{i_k})$ на куб $C$, заданный в координатах неравенствами

\begin{equation}\label{cube}
1 \le u_1 \le 2 \le u_2 \le 3  \dots \le {k-1} \le u_{k-1} \le k.
\end{equation}

Образно говоря, проекция $\pi$ состоит в запоминании у точки   $w\in GZ ( 1^{i_1}\,2^{i_2}\,\dots \,k^{i_k})$ только координат, стоящих во второй строке таблицы между разными числами первой строки, и забывании всех остальных.

\line(1,0){430}
\begin{equation}
\setcounter{MaxMatrixCols}{23}
\begin{matrix}
1&& \dots&&  1&& 2&& \dots&& {k-1}&& k&& \dots&& k\\
& 1& \dots& 1&& u_{1,i_1}&& 2& \dots& {k-1}&& u_{1,i_1+\dots i_{k-1}}&& k& \dots& k\\
&&&&& \|&&&&&& \|\\
&&&&& u_1&&&&&& u_{k-1}
\end{matrix}
\end{equation}

\line(1,0){430}
\begin{figure}[h]
\center{\includegraphics[width=0.45\linewidth]{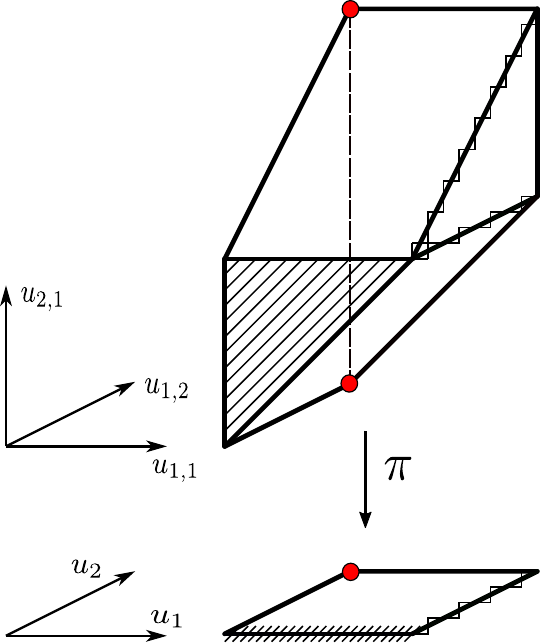}}
\caption{Проекция $\pi$ многогранника $GZ(1\, 2\, 3)$ на квадрат $C$. На рисунке выделены все вершины многогранника $GZ(1\, 2\, 3)$, которые отображаются на выделенную вершину квадрата; зигзагообразной линией выделены все рёбра, которые отображаются на ребро квадрата, помеченное зигзагообразной линией; штриховкой выделены все двумерные грани, которые отображаются на ребро квадрата, помеченное штриховкой.}
\label{projection}
\end{figure}

Сформулируем и докажем простую лемму.
\begin{lemma}\label{lem1}
Пусть $G$ -- непустая грань многогранника $GZ ( 1^{i_1}\,2^{i_2}\,\dots \,k^{i_k}) $, тогда $\pi(G)$ -- непустая грань куба $C$.
\end{lemma}
\begin{proof}
Нам уже известно, что образом всего многогранника  является весь куб, поэтому будем считать, что рассматриваемая грань -- собственная. Так как  $G$ -- непустая собственная грань многогранника $GZ( 1^{i_1}\,2^{i_2}\,\dots\,k^{i_k})$, то система определяющих её линейных неравенств и равенств получается из соответствующей треугольной таблицы заменой некоторых знаков неравенства на знаки равенства. При этом проекция $\pi$ забывает равенства и неравенства, находящиеся ниже второй строки таблицы, но оставляет неизменными неравенства и равенства, связывающие вторую и первую строки. Последние не могут носить противоречивый характер, так как иначе грань $G$ оказалась бы пустой. А значит они определяют некоторую грань куба $C$. Ясно также, что над любой точкой из полученной таким образом грани куба висит точка исходной грани $G$, то есть $\pi(G)$ -- непустая грань куба $C$.
\end{proof}

\begin{example}
Рассмотрим проекцию $\pi$ многогранника $GZ(1\, 2\, 3\,)$ на квадрат $C$, заданный в координатах неравенствами $1 \le u_1 \le 2 \le u_2 \le 3. $
 
Выпишем все двумерные грани многогранника $GZ(1\, 2\, 3\,)$, которые отображаются на  ребро квадрата $C$, заданное в координатах: $1\le u_1 \le 2,\ u_2=2$. Она одна (см. рис. \ref{projection}): $1\le u_{1,1} \le u_{2,1}\le 2,\ u_{1,2}=2$.

Выпишем все рёбра многогранника $GZ(1\, 2\, 3\,)$, которые отображаются на ребро квадрата $C$, заданное в координатах: $u_1=2,\  2\le u_2\le 3$. Их два (см. рис. \ref{projection}): $u_{1,1}=2,\  2\le u_{2,1}=u_{1,2}\le 3$ и $u_{1,1}=2,\  u_{2,1}=2,\  2\le u_{1,2}\le 3$.

Выпишем все вершины многогранника $GZ(1\, 2\, 3\,)$, которые отображаются на вершину $(1,\, 3)$ квадрата $C$. Их две (см. рис. \ref{projection}): $(1,\, 3,\, 1)$ и $(1,\, 3,\, 3)$.
\end{example}

Теперь мы знаем, что при проекции $\pi$ над каждой гранью куба $C$ висит некоторое количество граней исходного многогранника Гельфанда--Цетлина. Идея состоит в том, чтобы собрать искомый $f$-многочлен многогранника $GZ(1^{i_1}\,2^{i_2}\,\dots\,k^{i_k})$ из кусочков, соответствующих граням куба $C$. Как мы это будем делать? 

Пусть $A$ -- некоторая грань куба $C$, обозначим через $\widehat A$  барицентр грани $A$. Заметим, что между множеством барицетров всех граней куба $C$ и множеством мономов многочлена $p:= \prod_{j=1}^{k-1}(x_j+x_{j+0.5}+x_{j+1} ) $ существует взаимнооднозначное соответствие. Поясним это соответствие, практически дословно повторив рассуждение из статьи \cite{GKT}. В самом деле, чтобы зафиксировать барицетр $\widehat A$  некоторой грани куба $C$ нужно для каждого $j$ от $1$ до $k-1$ указать, какое из равенств $u_j=j, \ u_j=j+0.5 \ $  или $ \  u_j=j+1$ имеет место. Аналогично, чтобы зафиксировать моном в многочлене $ p $ нужно для каждого $j$ от $1$ до $k-1$ указать, какое слагаемое из скобки $( x_j+x_{j+0.5}+x_{j+1})$ войдёт в моном.

\begin{figure}[h]
\begin{minipage}[h]{0.49\linewidth}
\center{\includegraphics[width=0.85\linewidth]{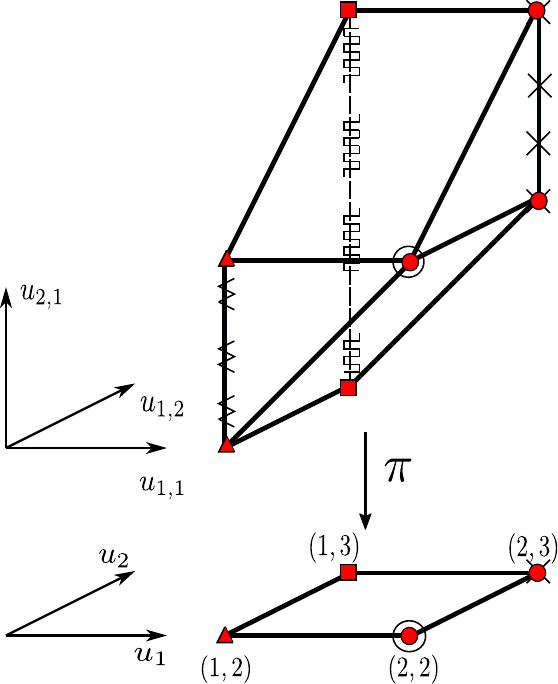} \\ а)}
\end{minipage}
\hfill
\begin{minipage}[h]{0.49\linewidth}
\center{\includegraphics[width=0.55\linewidth]{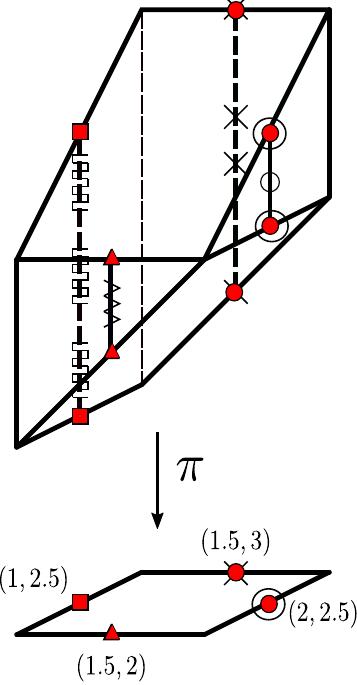} \\ б)}
\end{minipage}
\caption{Прообразы барицентров вершин и рёбер квадрата $C$ при проекции $\pi$ многогранника $GZ(1\,2\,3)$ на $C$.}
\label{pre1}
\end{figure}

\begin{example}\label{preimage}
Снова рассмотрим проекцию $\pi$ многогранника $GZ(1\, 2\, 3\,)$ на квадрат $C$, заданный в координатах неравенствами $1 \le u_1 \le 2 \le u_2 \le 3$. В этом случае $p=( x_1+x_{1.5}+x_2)(x_2+x_{2.5}+x_3)$. На рис. \ref{pre1} изображены полные прообразы барицентров вершин и рёбер квадрата $C$ при проекции $\pi$.
 
 Вершинам $C$ соответствуют мономы с целыми индексами $(1,2)\mapsto x_1 x_2$, $(1,3)\mapsto x_1 x_3$,  $(2,2) \mapsto x_2^2$, $(2,3) \mapsto x_2 x_3$, причём 
\[
\begin{matrix}
&& \xout1&&\xout2&&\xout3\\
\pi^{-1} (1,2)& \simeq&&1&&2&&=&GZ(1\,2),\\
&&&&u&\\
\end{matrix}
\] 
\[
\pi^{-1} (1,3)\simeq GZ(1\,3),\ \  \pi^{-1} (2,2) \simeq GZ(2^2),\ \  \pi^{-1} (2,3) \simeq GZ(2\,3).\ \ 
\] 

Рёбрам $C$ соответствуют мономы, в которых один из сомножителей имеет полуцелый индекс $(1.5,2)\mapsto x_{1.5} x_2$, $(1.5,3)\mapsto x_{1.5} x_3$, $(1,2.5) \mapsto x_1 x_{2.5}$, $(2,2.5) \mapsto x_2 x_{2.5}$, причём
\[
\begin{matrix}
&& \xout1&&\xout2&&\xout3\\
\pi^{-1} (1.5,2)& \simeq&&1.5&&2&&=&GZ(1.5\ 2),\\
&&&&u&\\
\end{matrix}
\]

\[
\pi^{-1} (1.5,3) \simeq GZ(1.5\ 3),\ \ \pi^{-1} (1,2.5) \simeq GZ(1\ 2.5),\ \ \pi^{-1} (2,2.5) \simeq GZ(2\ 2.5).\ \ 
\]

Единственной двумерной грани соответствует моном, в котором оба сомножителя имеют полуцелый индекс
 $(1.5,2.5)\mapsto x_{1.5} x_{2.5}$, причём
\[\begin{matrix}
&& \xout1&&\xout2&&\xout3\\
\pi^{-1} (1.5,2.5)& \simeq&&1.5&&2.5&&=&GZ(1.5\ 2.5).\\
&&&&u&\\
\end{matrix}\]
\end{example}
Теперь вернёмся к рассмотрению общего случая. Пусть $\widehat A$ -- барицентр, соответствующий моному

\begin{equation}\label{monom}
x_1^{\alpha_1} x_{1.5}^{\alpha_{1.5}}\cdots  x_{k-0.5}^{\alpha_{k-0.5}} x_k^{\alpha_k}.
\end{equation}

Тогда $\pi^{-1}( \widehat A )$ комбинаторно эквивалентен 

\begin{equation}\label{main row}
GZ( 1^{i_1-1+\alpha_1} \ 1.5^{\alpha_{1.5}}\ \dots  \ {\left(k-0.5\right)}^{\alpha_{k-0.5}} \ k^{i_k-1+\alpha_k}).
\end{equation}

Почему это так? Из примера \ref{preimage} становится ясно, что $\pi^{-1}( \widehat A )$ комбинаторно эквивалентен многограннику Гельфанда--Цетлина, первая строка которого может быть получена в два шага.
На первом шаге мы вносим в неё $i_q-1$ символов $q$ для каждого $q$ от $1$ до $k$ и $(k-1)$ пустой кружок так, как показано на схеме (\ref{second row}).

\begin{equation}\label{second row}
\underbrace{1\dots 1}_{i_1-1}\bigcirc \underbrace{2\dots 2}_{i_2-1} \bigcirc \dots \bigcirc \underbrace{k \dots k}_{i_k-1}
\end{equation}

Мультипликативный вариант картинки выглядит так

\[
1^{i_1-1}\bigcirc 2^{i_2-1} \bigcirc \dots \bigcirc k^{i_k-1}.
\]

А затем заполняем пустые кружочки координатами точки $\widehat A$, двигаясь слева направо. При этом показатель $\alpha_r$, где $r\in \{1,\, 1.5,\, 2,\dots,\,k\}$ в мономе (\ref{monom}) говорит нам, сколько чисел $r$ мы должны использовать.

Итак, теперь мы знаем, что прообразы барицентров граней куба $C$ являются многогранниками Гельфанда--Цетлина, более того для каждого прообраза мы знаем определяющую его неубывающую последовательность действительных чисел  $$ 1^{i_1-1+\alpha_1} \ 1.5^{\alpha_{1.5}}\ \dots  \ {(k-0.5)}^{\alpha_{k-0.5}} \ k^{i_k-1+\alpha_k}.$$ Покажем теперь, что проекция $\pi$ позволяет выписать рекуррентное соотношение, связывающее $f$-многочлен исходного многогранника Гельфанда--Цетлина с $f$-много\-членами прообразов барицентров граней куба $C$.

Пусть $G$ -- некоторая грань многогранника $ GZ( 1^{i_1} \dots k^{i_k}) $, обозначим через $A_G$ грань куба $C$, которая является образом грани $G$ при проекции $\pi$ (здесь мы пользуемся леммой \ref{lem1}), тогда $\widehat {A_G}$ -- барицентр грани $A_G$. Пересечение многогранника $\pi^{-1} ( \widehat {A_G} )$ и грани $G$ является некоторой гранью этого многогранника, которую мы обозначим через $E_G$.

\begin{figure}[h]
\center{\includegraphics[width=0.45\linewidth]{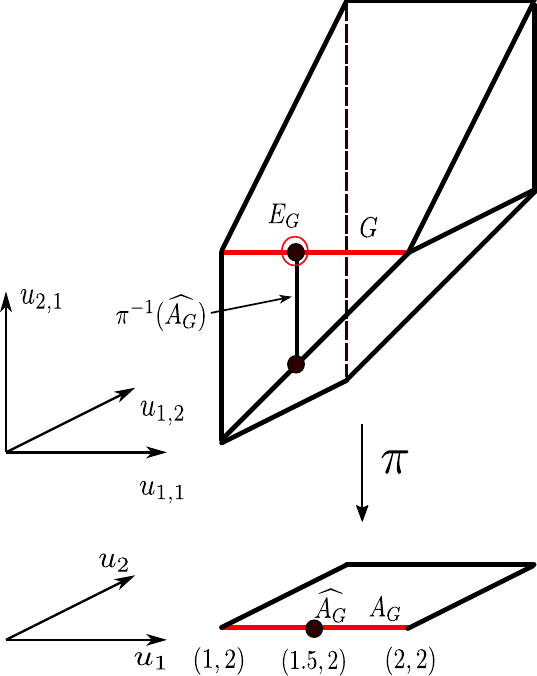}}
\caption{Иллюстрация к примеру \ref{geomlemma}.}
\label{figexample}
\end{figure}

\begin{example}\label{geomlemma}
	Рассмотрим многогранник $GZ(1\,2\,3\,)$, рис. \ref{figexample}. Пусть грань $G$ задана в координатах:  $u_{1,2}=u_{2,1}=2,\ 1 \le u_{1,1} \le2$,-- найдём $E_G$. Сначала выпишем условия, определяющие $A_G$: $ 1 \le u_1 \le2,\  u_2=2$, теперь ясно, что барицетр $\widehat {A_G}$ -- это точка $(1.5,2)$. Следующим шагом выпишем условия, определяющие полный прообраз барицентра $\pi^{-1} (\widehat{A_G} )$:  $u_{1,1}=1.5,\ u_{1,2}=2,\ 1.5 \le u_{2,1} \le 2$. Теперь ясно, что  искомое пересечение  $E_G$ -- это точка $(1.5,2,2)$.
	
\end{example}
Пользуясь введёнными обозначениями и приобретённой в результате рассмотрения примера интуицией, сформулируем и докажем лемму \ref{main lemma}.

\begin{lemma}\label{main lemma}
Отображение  из множества всех непустых граней многогранника $ GZ( 1^{i_1} \dots k^{i_k}) $ во множество всех непустых граней многогранников $\pi^{-1}( \widehat A )$, которое переводит грань $G$  в грань $E_G$ является биекцией, причём оно не сохраняет размерность грани, но изменяет её понятным образом, а именно 
\begin{equation}\label{dimension}
\dim\,G=\dim\,E_G+\dim\,A_G.
\end{equation}
\end{lemma} 

Мы приведём доказательство леммы \ref{main lemma}, которое опирается на определения \textit{симплициального комплекса} и \textit{производного измельчения многогранника}. Напомним эти определения для полноты картины.

\begin{definition}
Сиплициальный комплекс $K$ -- это такой конечный набор симплексов, лежащих в некотором евклидовом пространстве, что
\begin{enumerate}
\item любая грань симплекса из $K$ снова принадлежит $K$;
\item пересечение любых двух симплексов $\sigma_1, \, \sigma_2 \in K$ либо пусто, либо является гранью как симплекса $\sigma_1$, так и симплекса $\sigma_2$.
\end{enumerate}
Множество $|K|=\bigcup_{\sigma \in K} \sigma$ называется телом комплекса $K$.
\end{definition}

\begin{definition}\label{bar subdivision}
Производное измельчение многогранника $P$ -- это симплициальный комплекс $\Delta^P$, описание которого мы дадим с помощью индукции по размерности граней многогранника $P$.
\begin{itemize}
\item (база индукции) Каждая нульмерная грань многогранника $P$ совпадает со своим производным  измельчением, поэтому симплексами комплекса $\Delta^P$, полученными на первом шаге, являются вершины многогранника $P$. 
\item (шаг индукции) Предположим, что мы построили производное измельчение всех граней $F\in L(P)$ размерности от $0$ до $k-1$. Для каждой $k$-мерной грани многогранника $P$ мы делаем следующее: во внутренности грани фиксируем некоторую точку, а затем берем выпуклую оболочку зафиксированной точки с каждым из $(k-1)$-мерных симплексов, лежащих в границе рассматриваемой грани. Заметим, что выпуклая оболочка $(k-1)$-мерного симплекса и точки, не принадлежащей его аффинной оболочке, есть в точности $k$-мерный симплекс. Все грани полученных таким образом $k$-мерных симплексов, которые ещё не были включены в $\Delta^P$, на этом шаге добавляются.
 
\end{itemize}
\end{definition}
Непосредственно из определения \ref{bar subdivision} следует, что множество вершин производного измельчения взаимнооднозначно соответствует множеству граней исходного многогранника.

\begin{remark}
Если при построении производного измельчения многогранника фиксировать во внутренности каждой грани не произвольную точку, а барицентр соответствующей грани, то получаемый симплициальный комплекс называют барицентрическим подразбиением (или барицентрическим измельчением) исходного многогранника. Таким образом, барицентрическое подразбиение -- это частный случай производного измельчения.
\end{remark}

Теперь мы готовы доказать лемму \ref{main lemma}.

\begin{proof}	
Рассмотрим производное измельчение куба $C$, вершины которого являются барицентрами граней, а также производное измельчение многогранника Гельфанда--Цетлина, вершины которого лежат в прообразах барицентров граней куба $C$ при проекции $\pi$. По определению производного измельчения, во внутренности каждой грани $G$ исходного многогранника Гельфанда--Цетлина есть единственная вершина производного измельчения, и она по построению лежит в грани $E_G$. Поскольку она лежит в грани $E_G$  и во внутренности грани $G$, то она лежит во внутренности грани $E_G$. Причём верно, что такая вершина только одна, потому что внутренность $E_G$ высекается из внутренности единственной грани исходного многогранника, а именно из внутренности грани $G$. Таким образом, и грани исходного многогранника, и грани многогранников $\pi^{-1}(\widehat A)$ оказались занумерованы вершинами производного измельчения исходного многогранника, причём грани $G$  и $E_G$ имеют один и тот же номер.

Формула (\ref{dimension}), связывающая размерности граней $G$ и $E_G$, может быть выписана сразу, как только мы вспомним,что размерность образа конечномерного векторного пространства при линейном отображении равна разности размерностей пространства и ядра отображения.
\end{proof}

\begin{corollary}\label{main corollary}
Можно восстановить $f$-многочлен многогранника $GZ( 1^{i_1} \dots k^{i_k})$ по $f$-многочленам многогранников $\pi^{-1}(\widehat A )$ при помощи равенства 
	
\begin{equation}\label{main formular}
f_{GZ}(t)=\sum_{A \in L(C)}t^{\dim\, A}f_{\pi^{-1}(\widehat A )}(t).
\end{equation}

\end{corollary}

\begin{example}
	Пользуясь формулой (\ref{main formular}) и примером \ref{preimage}, выпишем $f$-многочлен многогранника $GZ(1\, 2\,3)$.
\begin{equation*}
	\begin{split}
	f_{GZ(1\,2\,3)}&=f_{GZ(1\ 2)}+f_{GZ(1\ 3)}+f_{GZ(2^2)}+f_{GZ(2\ 3)}+\\
	&+t \cdot ( f_{GZ(1.5\ 2)} +f_{GZ(1.5\ 3)}+f_{GZ(1\ 2.5)}+f_{GZ(2\ 2.5)})+\\
	&+t^2 \cdot f_{GZ(1.5\  2.5)}
	\end{split}
\end{equation*}
В рассматриваемом случае один из прообразов нульмерен, а все остальные комбинаторно эквивалентны отрезку, поэтому искомый $f$-многочлен равен
\begin{equation*}
	\begin{split}
		f_{GZ(1\,2\,3)}(t)&=(2+t)+(2+t)+1+(2+t)+\\
		&+t \cdot \big( (2+t) +(2+t)+(2+t)+(2+t)\big)+\\
		&+t^2 \cdot (2+t)=\\
		&=7+11t+6t^2+t^3.
	\end{split}
\end{equation*}

То есть многогранник $GZ(1\,2\,3)$ имеет $7$ вершин, $11$ рёбер и $6$ двумерных граней.
\end{example}

Настало время объединить полученные знания в одну формулу. Вспомним, что мы занумеровали барицентры граней куба $C$ мономами многочлена $p:= \prod_{j=1}^{k-1}(x_j+x_{j+0.5}+x_{j+1} )$. Определим коэффициенты $c_{\alpha_1,\alpha_{1.5}, \dots, \alpha_{k-0.5}, \alpha_k}$ таким образом, что
\begin{equation}
p=\sum_{\alpha_1,\alpha_{1.5}, \dots, \alpha_{k-0.5}, \alpha_k} c_{\alpha_1,\alpha_{1.5}, \dots, \alpha_{k-0.5}, \alpha_k} \, x_1^{\alpha_1} x_{1.5}^{\alpha_{1.5}}\cdots  x_{k-0.5}^{\alpha_{k-0.5}} x_k^{\alpha_k}.
\end{equation}
Итоговая формула будет более читаемой, если мы будем использовать мульти-индексы  $\alpha=(\alpha_1,\alpha_{1.5}, \dots, \alpha_{k-0.5}, \alpha_k)$ и $i=(i_1, i_2, \dots, i_k)$ для её записи. Обозначим через $GZ[\alpha, \,i]$ многогранник Гельфанда--Цетлина, комбинаторно эквивалентный $\pi^{-1}(\widehat A_{\alpha})$, явный вид которого мы уже выписали формулой (\ref{main row}), тогда имеет место 
\begin{theorem}\label{recurrence relation}
Справедливо следующее рекуррентное соотношение на 
$f$-многочлен многогранника Гельфанда--Цетлина $GZ(1^{i_1} \dots k^{i_k})$
\begin{equation}
	f_{GZ(1^{i_{^1}} \dots \, k^{i_{^k}})}(t)=\sum_{\alpha} c_{\alpha} \cdot t^{\dim (A_{\alpha})} \cdot f_{GZ[\alpha,\,i]}(t).
\end{equation}
\end{theorem}

\section{Вычисления с использованием рекуррентного соотношения} \label{Calculate}
Покажем, что найденное рекуррентное соотношение позволяет вычислить $f$-мно\-гочлен для некоторых серий многогранников Гельфанда--Цетлина. А именно, рассмотрим однопараметрические семейства, соответствующие неубывающим последовательностям, содержащим три различных числа:  $GZ(1\,2^k \, 3)$ и  $GZ(1\, 2\, 3^k)$. (Семейство $GZ(1^k\, 2\, 3)$ в силу леммы \ref{onelemma} не даёт нам нового знания, поэтому его мы отдельно не рассматриваем.)

\begin{lemma}\label{onelemma}
	Если между наборами чисел $\lambda_1 \le \dots \le \lambda_s$ и $\lambda_1^{\prime} \le \dots \le \lambda_s^{\prime}$ существует биекция, сохраняющая отношение $=$, но обращающая строгий порядок, то соответствующие многогранники Гельфанда--Цетлина комбинаторно эквивалентны.
\end{lemma}
\begin{proof} 
	Рассмотрим многогранники $GZ(\lambda_1 \dots \lambda_s)$ и $GZ({-\lambda_s}  \dots {-\lambda_1})$, покажем, что они аффинно (и, следовательно, комбинаторно) эквивалентны.
		
	Будем считать, что координаты точек пространства $\R^ \frac{s(s-1)}{2}$ занумерованы так же, как в треугольной таблице (\ref{tab}), то есть двойным индексом $(i,j)$, где $1 \le i \le s-1$, $1 \le j \le s-i$. Рассмотрим линейное отображение $F$, переводящее точку с координатами $u_{i,j}$ в точку с координатами $u'_{i,j}$, где $u'_{i,j}=-u_{i,(s-i)-(j-1)}$. (Поясним происходящее с точки зрения геометрии треугольной таблицы: $F$ отражает треугольную таблицу относительно вертикальной оси и обращает все знаки.) Найдём образ многогранника $GZ(\lambda_1 \dots  \lambda_s)$ при линейном отображении $F$. Несложно показать, что линейные неравенства, связывающие первую и вторую строку треугольной таблицы, определяющей многогранник  $GZ(\lambda_1 \dots  \lambda_s)$, перейдут в неравенства вида  $-\lambda_{s-j+1} \le u'_{1,j} \le -\lambda_{s-j}$ (здесь $j$ пробегает значения от $1$ до $s-1$), а  остальные неравенства -- в неравенства вида  $u'_{i-1,j} \le u'_{i,j} \le u'_{i-1,j+1}$ (здесь $i$ пробегает значения от $2$ до $s-1$ и $j$ пробегает значения от $1$ до $s-i$). Таким образом, $F\big(GZ(\lambda_1 \dots \lambda_s)\big)=GZ({-\lambda_s}  \dots {-\lambda_1})$. Ясно, что линейное отображение $F$ обратно само себе и, следовательно, обратимо. Значит многогранники $GZ(\lambda_1 \dots \lambda_s)$ и $GZ({-\lambda_s}  \dots {-\lambda_1})$ аффинно (и, следовательно, комбинаторно) эквивалентны.
	
	Обозначим через $\phi$ биекцию  из  $\lambda_1^{\prime}\le \dots \le \lambda_s^{\prime}$ в  $\lambda_1 \le \dots \le \lambda_s$, сохраняющую отношение $=$, но обращающую строгий порядок, и обозначим через $\psi$ биекцию между наборами чисел $\lambda_1  \le \dots \le \lambda_s$ и $-\lambda_s \le  \dots \le -\lambda_1$, при которой $\lambda_j \mapsto -\lambda_j$. Тогда композиция $\psi \circ \phi$ будет биекцией из $\lambda_1^{\prime} \le \dots \le \lambda_s^{\prime}$ в  $-\lambda_s  \le \dots \le -\lambda_1$, сохраняющей отношения $=$ и $<$, как следствие 
	$GZ(\lambda_1^{\prime} \dots \lambda_s^{\prime}) \simeq GZ({-\lambda_s}  \dots {-\lambda_1})$. Ранее мы показали, что $GZ(\lambda_1 \dots \lambda_s) \simeq GZ({-\lambda_s}  \dots {-\lambda_1})$. Так как отношение эквивалентности транзитивно, то $GZ(\lambda_1^{\prime} \dots \lambda_s^{\prime}) \simeq GZ(\lambda_1  \dots \lambda_s)$.	
\end{proof}

\subsection{Вычисление $f$- и $h$-многочлена многогранников $GZ(1\  2^k\   3)$}\label{subsecGZ123k}
Рассмотрим серию многогранников $GZ(1\,2^k \, 3)$. В этом случае куб $C$, как и в примере \ref{preimage}, может быть задан в координатах неравенствами $1 \le u_1 \le 2 \le u_2 \le 3$.

Многогранники, комбинаторно эквивалентные полным прообразам барицентров граней куба $C$, в рассматриваемом случае  получаются очень просто. Пусть барицентр $\widehat A$  некоторой грани куба $C$ имеет координаты $(u_1,u_2)$, тогда $\pi^{-1}(\widehat A)\simeq GZ(u_1 \ 2^{k-1}\ u_2)$. К примеру, многогранник, комбинаторно эквивалентный полному прообразу барицентра двумерной грани, оказывается таким: 
$$\pi^{-1}(1.5,2.5) \simeq GZ(1.5\ 2^{k-1}  \, 2.5 ) \simeq GZ(1\ 2^{k-1}\, 3).$$
(Для второго перехода мы используем достаточно очевидное замечание \ref{AlmostOb}.) Опуская технические детали, связанные с подстановкой полученных данных в формулу (\ref{main formular}) и приведением подобных слагаемых ($f$-многочленов $GZ(1\ 2^{k-1}\, 3)$, а также $f$-многочленов симплексов $GZ(1\ 2^k)$ и $GZ(2^k\ 3)$, см. пример \ref{simplexGZ}), выпишем полученное рекуррентное соотношение для $f$-многочлена многогранника $GZ(1\ 2^{k}\ 3)$:

\begin{equation} \label{rec123}	
		f_{GZ(1\ 2^{k}\ 3)}(t)= (1+2t+t^2)\cdot f_{GZ(1\ 2^{k-1} \, 3)}+(2+2t)\cdot f_{GZ(1\ 2^k)}+1.	
\end{equation}
Будем решать его с учётом того, что $f$-многочлен симплекса $GZ(1\ 2^k)$ нам известен. Итак, соотношение (\ref{rec123}) имеет вид $b_k=\mu b_{k-1}+\nu a_k +1,$ где $b_k=f_{GZ(1\ 2^{k}\ 3)}$ и $a_k=f_{GZ(1\ 2^k)}$ -- это $f$-многочлены, а $\mu=(1+t)^2$ и $\nu=2+2t$ -- это коэффициенты, зависящие от $t$. Его общее решение имеет вид:

\begin{equation} \label{general solution}
	b_k=\mu^k b_0 + \sum_{j=1}^k \mu^{k-j}(\nu a_j +1).
\end{equation}
Так как $b_0=f_{GZ(1\, 3)}$ и $a_j=\frac{(1+t)^{j+1}-1}{t}$, то общее решение рекуррентного соотношения (\ref{rec123}) имеет вид:
\begin{equation} \label{GZ(12k3)}
f_{GZ(1\ 2^{k}\ 3)}(t)=(1+t)^{2k}\cdot(2+t)+\sum_{j=1}^k(1+t)^{2(k-j)}\left((2+2t)\cdot \frac{(1+t)^{j+1}-1}{t} +1\right).
\end{equation}
В нашем случае удобно перейти к многочлену $h_{GZ(1\ 2^k\ 3)}(s)=f_{GZ(1\ 2^k\ 3)}(s-1)=\sum_{i=0}^{2k+1}h_is^i$:  

\begin{equation}\label{h123k}
h_{GZ(1\ 2^{k}\ 3)}(s)=s^{2k}\cdot(1+s)+\sum_{j=1}^k s^{2(k-j)}\left(2s\cdot \frac{s^{j+1}-1}{s-1} +1\right).
\end{equation}
\begin{theorem}\label{hvector12k3}
 Компоненты $h$-вектора многогранника $GZ(1\, 2^{k}\,3)$, имеющего размерность $2k+1$, имеют вид
$$h_{2k+1}=1,$$
 \begin{equation}\label{h123kcases} h_{i}=
	 \begin{cases}
	 i+1, \, \textit{если}\  0\le i \le k+1;\\
	 2k+3-i, \textit{если}\  k+1 < i \le 2k.
	 \end{cases}
 \end{equation} 
\end{theorem}
\begin{proof}
Перепишем многочлен (\ref{h123k}) так, чтобы было проще вычислить его коэффициенты
\begin{equation}\label{counting formular}
	h_{GZ(1\ 2^{k}\ 3)}(s)=s^{2k+1}+ \sum_{j=0}^ks^{2k-2j}+2\cdot \sum_{j=1}^{k}(s^{2k+1-j}+ \dots + s^{2k+1-2j}).
\end{equation}
Рассмотрим подробнее сумму $\sum_{j=1}^{k}(s^{2k+1-j}+ \dots + s^{2k+1-2j})$. При фиксированном значении  $j$ слагаемые внутри соответствующей скобки -- это степени $s$, показатели которых убывают от $2k+1-j$ до $2k+1-2j$. При увеличении значения  $j$ на единицу количество слагаемых в скобке возрастает на один, при этом наибольший показатель убывает на единицу, а наименьший убывает на две единицы. Это наблюдение даёт возможность представлять себе рассматриваемую сумму в виде ступеней, каждая из которых соответствует конкретному значению $j$ (см. пример \ref{h(12^53)}). При этом для каждого нечётного числа от $1$ до $2k-1$ и для каждого чётного числа от $2$ до $2k$ несложно вычислить значение $j'$, при котором оно возникает в качестве показателя степени впервые (в случае нечётного -- у последнего слагаемого, а в случае чётного -- у предпоследнего слагаемого в некоторой скобке). Далее заметим, что, степень, появившись в некоторой скобке впервые, в каждой следующей скобке продвигается на одну позицию ближе к началу, причём возможны два варианта:
\begin{itemize}
	\item при некотором значении $j''$ она становится старшим членом в скобке и (если $j''\ne k$)  в последующих скобках уже не появляется,
	\item присутствует во всех последующих скобках (быть может, став старшим членом в последней, соответствующей значению $j=k$).
\end{itemize}
Так как коэффициенты степеней $s$ в каждой скобке равны единице, то для вычисления коэффициента перед заданной степенью достаточно посчитать количество скобок, в которых она встречается. Рассмотрим отдельно случаи чётного и нечётного показателя. 

Найдём $j'$ и $j''$ для степени $s^{2k-2l}$. Так как $2k+2-2j'=2k-2l$, то $j'=l+1$. Так как $2k+1-j''=2k-2l$, то $j''=2l+1$.
 
Найдём $j'$ и $j''$ для степени $s^{2k-2l-1}$. Так как $2k+1-2j'=2k-2l-1$, то $j'=l+1$. Так как $2k+1-j''=2k-2l-1$, то $j''=2l+2$.

Теперь мы можем вычислить коэффициент перед заданными степенями:
\begin{equation*}
[\min\{2l+1,k\}-l]\, s^{2k-2l},\ 
[\min\{2l+2,k\}-l]\, s^{2k-2l-1}.
\end{equation*}
Подставляя полученные формулы в многочлен (\ref{counting formular}), получим 

$h_{2k+1}=1$, 

$h_{2k-2j}=1+2(\min\{2j+1,k\}-j)$ при $0\le j\le k$,

$h_{2k-2j-1}=2(\min\{2j+2,k\}-j)$ при $0 \le j \le k-1$.
Далее непосредственное вычисление показывает нам, что компоненты $h$-вектора многогранника $GZ(1\, 2^{k}\,3)$, неформально говоря, образуют несимметричную горку, вершиной которой (наибольшей компонентой вектора) является $h_{k+1}$ и справедливы формулы (\ref{h123kcases}).
\end{proof}

\begin{example}\label{h(12^53)}
	Положим $k=5$ и найдём $h$-многочлен и $h$-вектор многогранника $GZ(1\, 2^5\, 3)$.
	При заданном значении $k$ формула (\ref{counting formular}) принимает вид (часть формулы разрезана на ступени, соответствующие значениям $j$ от 1 до 5):
	\medskip
	\begin{equation*}
	\begin{aligned}		
	  h_{GZ(1\ 2^5\ 3)}(s)=s^{11}+ \smash{\sum_{j=0}^5}s^{10-2j}+2\cdot
	  (s^{10}&+s^9&\!\!\!\!+\,&\\
	         &+s^9&\!\!\!\!+\,&s^8&\!\!\!\!+\,&s^7&\!\!\!\!+\,&\\
   	         &    &\!\!\!\!+\,&s^8&\!\!\!\!+\,&s^7&\!\!\!\!+\,&s^6+s^5&\!\!\!\!+\,&\\
	         &    &           &   &\!\!\!\!+\,&s^7&\!\!\!\!+\,&s^6+s^5&\!\!\!\!+\,&s^4+s^3+&\\
	         &    &           &   &           &   &\!\!\!\!+\,&s^6+s^5&\!\!\!\!+\,&s^4+s^3+&\!\!\!\!s^2+s).\\
	\end{aligned}
	\end{equation*}
	Мы получили многочлен $11$-ой степени, коэффициенты которого легко считываются по столбцам одинаковых степеней с учётом того, что сумма $$\sum_{j=0}^5s^{10-2j}= s^{10}+s^8+s^6+s^4+s^2+1$$ делает вклад равный единице в каждый коэффициент, стоящий при степени с чётным показателем.
	Итак,
	\begin{equation*}
	\begin{split}
	h_{GZ(1\ 2^5\ 3)}(s)=\,&s^{11}+3s^{10}+4s^9+5s^8+6s^7+7s^6+\\
						+\,&6s^5+5s^4+4s^3+3s^2+2s+1
	\end{split}
	\end{equation*}
	и соответствующий $h$-вектор $h_{GZ(1\,2^5\, 3)}=(1,2,3,4,5,6,7,6,5,4,3,1)$ c наибольшей компонентной $h_{5+1}=h_6=7$.
\end{example}

\subsection{Вычисление $f$- и $h$-многочленов многогранников $GZ(1\, 2\, 3^k)$ и $GZ(2^2\, 3^k)$ }\label{subsecGZ123k223k}
Теперь рассмотрим серию многогранников $GZ(1\, 2\, 3^k)$. В этом случае куб $C$ также может быть задан в координатах неравенствами $1 \le u_1 \le 2 \le u_2 \le 3$. Пусть барицентр $\widehat A$ некоторой грани куба $C$ имеет координаты $(u_1,u_2)$, тогда $\pi^{-1}(\widehat A)\simeq GZ(u_1 \ u_2\ 3^{k-1})$. К примеру, многогранник, комбинаторно эквивалентный полному прообразу барицентра двумерной грани, оказывается таким:
$$\pi^{-1}(1.5,2.5) \simeq GZ(1.5\ 2.5\ 3^{k-1} ) \simeq GZ(1\ 2\ 3^{k-1}).$$

Рекуррентное соотношение (\ref{main formular}) для $f$-многочлена многогранника $GZ(1\, 2\, 3^k)$ после приведения подобных слагаемых выглядит следующим образом:
\begin{equation} \label{GZ(123k)}
		f_{GZ(1\ 2\ 3^k)}(t)=f_{GZ(2^2\ 3^{k-1})}+(1+3t+t^2)\cdot f_{GZ(1\, 2\ 3^{k-1})}+(2+t)\cdot f_{GZ(1\ 3^k)}.
\end{equation}

В (\ref{GZ(123k)}) кроме многогранников нашей серии и симплексов Гельфанда--Цетлина появился многогранник из серии $GZ(2^2\ 3^k)$. Пользуясь формулой (\ref{main formular}), напишем и для многогранников этой серии рекуррентное соотношение. В этом случае куб $C$ -- это просто отрезок, он может быть задан в координатах неравенствами $2 \le u_1 \le 3$. Пусть барицентр $\widehat A$ некоторой грани куба $C$ имеет координату $u_1$, тогда $\pi^{-1}(\widehat A)\simeq GZ(2 \ u_1\ 3^{k-1})$. 

Иcкомое рекуррентное соотношение для  $f$-многочлена многогранника $GZ(2^2 3^k)$ имеет вид:
\begin{equation} \label{GZ(223k)}
f_{GZ(2^2\ 3^k)}(t)=f_{GZ(2^2\ 3^{k-1})}+ f_{GZ(2\ 3^k)} + t\cdot f_{GZ(1\, 2\ 3^{k-1})}.
\end{equation}

Соотношения (\ref{GZ(123k)}) и (\ref{GZ(223k)}) образуют систему: 
\begin{equation}\label{system} 
	\begin{cases}
	f_{GZ(1\ 2\ 3^k)}(t)= (1+3t+t^2)\cdot f_{GZ(1\, 2\ 3^{k-1})} + f_{GZ(2^2\ 3^{k-1})}+(2+t)\cdot f_{GZ(1\ 3^k)},\\
	f_{GZ(2^2\ 3^k)}(t)= t\cdot f_{GZ(1\, 2\ 3^{k-1})} + f_{GZ(2^2\ 3^{k-1})} + f_{GZ(2\ 3^k)}.
	\end{cases}
\end{equation}
Введём дополнительные обозначения: $d_k = f_{GZ(1\, 2\, 3^k)}$ , $b_k =f_{GZ(2^2\, 3^k)}$, $a_k=f_{GZ(1\, 3^k)}$,  $\mu_1 = 1+3t+t^2$, $\mu_2 = t$, тогда система (\ref{system}) примет вид:

\begin{equation*} 
	\begin{cases}
	d_k=\mu_1\cdot d_{k-1}+b_{k-1}+(2+t) a_k\\
	b_k=\mu_2\cdot d_{k-1}+b_{k-1}+a_k,
	\end{cases}
\end{equation*}
найдём её общее решение  в матричной форме:
\begin{equation*}
	\begin{pmatrix} 
	d_k\\
	b_k 
	\end{pmatrix} =	
	\begin{pmatrix}
	 \mu_1& 1\\
	 \mu_2& 1
	\end{pmatrix}^k 
	\begin{pmatrix}
	d_0\\
	b_0
	\end{pmatrix} +
	\sum_{j=1}^{k}
	\begin{pmatrix}
		 \mu_1& 1\\
		 \mu_2& 1
	\end{pmatrix}^{k-j}
		\begin{pmatrix}
			(2+t)a_j\\
			a_j
		\end{pmatrix}.
\end{equation*}
Так как $d_0=f_{GZ(1\, 2\, 3^0)}=2+t$ и $b_0=f_{GZ(2^2\, 3^0)}=1$, то общее решение системы (\ref{system}) в матричной форме имеет вид:
\begin{equation}
	\begin{pmatrix}
	f_{GZ(1\, 2\, 3^k)}(t)\\
	f_{GZ(2^2\, 3^k)}(t)
	\end{pmatrix} =
	M_t^k
	\begin{pmatrix}
	2+t\\
	1
	\end{pmatrix} + \sum_{j=1}^{k}
	M_t^{k-j} 
\begin{pmatrix}
(2+t)\cdot \frac{(1+t)^{j+1}-1}{t}\\
\frac{(1+t)^{j+1}-1}{t}
\end{pmatrix}, 	
\end{equation}
где
\[
M_t=\begin{pmatrix}
 (1+t)^2+t& 1\\
 t& 1
\end{pmatrix}. 
\]
Для дальнейших вычислений нам будет удобно перейти к $h$-многочленам многогранников $GZ(1\, 2\, 3^k)$ и $GZ(2^2\, 3^k)$
\begin{equation}\label{sysh}
	\begin{pmatrix}
	h_{GZ(1\, 2\, 3^k)}(s)\\
	h_{GZ(2^2\, 3^k)}(s)
	\end{pmatrix} =
	M_s^k
	\begin{pmatrix}
	s+1\\
	1
	\end{pmatrix} + \sum_{j=1}^{k}
	M_s^{k-j} 
\begin{pmatrix}
(s+1)\cdot \frac{s^{j+1}-1}{s-1}\\
\frac{s^{j+1}-1}{s-1}
\end{pmatrix}, 	
\end{equation}
где
\[
M_s=\begin{pmatrix}
 s^2+s-1& 1\\
 s-1& 1
\end{pmatrix}. 
\]
\begin{theorem}\label{hpolynom123k223k}
Пусть $\{\Phi_k\}$ последовательность,  заданная рекуррентным соотношением $\Phi_{k+1}=(s^2+s)\Phi_k-s^2\Phi_{k-1}$ с начальными условиями $\Phi_0=0$ и $\Phi_1=1$, тогда

A)формула для $h$-многочлена многогранника $GZ(1\, 2\, 3^k)$ имеет вид
\begin{equation}\label{h(123^k)}
h_{GZ(1\, 2\, 3^k)}(s)=\sum_{j=0}^{k}\frac{s^{j+2}-1}{s-1}\Phi_{k-j+1} ;
\end{equation}

Б) формула для $h$-многочлена многогранника $GZ(2^2\, 3^k)$ имеет вид
\begin{equation}\label{h(223^k)}
h_{GZ(2^2\,3^k)}(s)=\sum_{j=0}^{k}s^{j+2}\Phi_{k-j}+\frac{s^{k+1}-1}{s-1}.
\end{equation}
\end{theorem}
\begin{proof}
Приведём матрицу $M_s$ к фробениусовой нормальной форме, это упростит нам возведение в степень
\[
M_s=
\begin{pmatrix}
0& 1\\
1& 1
\end{pmatrix}
\begin{pmatrix}
0& -s^2\\
1& s^2+s
\end{pmatrix}
\begin{pmatrix}
-1& 1\\
1& 0
\end{pmatrix}. 
\]
Несложно доказать по индукции, что 
\begin{equation}\label{FKLk}
\begin{pmatrix}
0&-s^2\\
1& s^2+s
\end{pmatrix}^k=
\begin{pmatrix}
-s^2\cdot \Phi_{k-1}& -s^2\cdot \Phi_k\\
\Phi_k&           \Phi_{k+1}
\end{pmatrix}.
\end{equation}
 Положим также $\Phi_{-1}=-\frac{1}{s^2}$, что обеспечит нам корректное возведение в нулевую степень. Вычислив производящую функцию 
\begin{equation}\label{Phi(z)}
\Phi(z)=\sum_{k=0}^{\infty}\Phi_kz^k=\frac{z}{s^2z^2-(s^2+s)z+1},
\end{equation}
  мы находим общее решение рекуррентного соотношения $\Phi_{k+1}=(s^2+s)\Phi_k-s^2\Phi_{k-1}$ при $k\ge 0$
\begin{equation}\label{explicit}
\Phi_k=\frac{1}{\sqrt{s^2+2s-3}}\cdot\Bigg(\bigg(\frac{s+1+\sqrt{s^2+2s-3}}{2}\bigg)^k-\bigg(\frac{s+1-\sqrt{s^2+2s-3}}{2}\bigg)^k\Bigg).
\end{equation}
Теперь вернёмся к возведению в степень матрицы $M_s$. Вычисления с использованием формулы (\ref{FKLk}) дают следующий результат
\begin{equation}\label{Mk}
M_s^k=
\begin{pmatrix}
0& 1\\
1& 1
\end{pmatrix}
\begin{pmatrix}
-s^2 \Phi_{k-1}& -s^2 \Phi_k\\
\Phi_k              & \Phi_{k+1}
\end{pmatrix}
\begin{pmatrix}
-1& 1\\
1& 0
\end{pmatrix}=
\begin{pmatrix}
\Phi_{k+1}-\Phi_{k}& \Phi_k \\
(s-1)\Phi_k        & \Phi_k-s^2\Phi_{k-1}
\end{pmatrix}. 
\end{equation}

Подставим (\ref{Mk}) в (\ref{sysh}). Непосредственно после подстановки мы получим
\begin{equation}\label{hfool(123^k)}
h_{GZ(1\, 2\, 3^k)}(s)=\sum_{j=   0}^{k}\big((s+1)\Phi_{k-j+1}-s\Phi_{k-j}\big)\frac{s^{j+1}-1}{s-1},
\end{equation}
\begin{equation}\label{hfool(223^k)}
h_{GZ(2^2\,3^k)}(s)=s^2\sum_{j=0}^{k}\big(\Phi_{k-j}-\Phi_{k-j-1}\big)\frac{s^{j+1}-1}{s-1}.
\end{equation}
Далее остаётся только привести подобные слагаемые.
Заметим в завершении доказательства, что мы доопределили $\Phi_{-1}=\frac{1}{A}=-\frac{1}{s^2}$, выйдя из кольца многочленов в поле рациональных функций, но $\Phi_{-1}$ появляется только в формуле (\ref{hfool(223^k)}) и совершенно не мешает $h_{GZ(2^2\, 3^k)}(s)$ быть многочленом.
\end{proof}

\begin{example}
Положим $k=3$ и найдём $h$-многочлен и $h$-вектор для многогранников $GZ(1\,2\, 3^3)$ и $GZ(2\, 2\, 3^3)$. При заданном значении $k$ формула (\ref{h(123^k)}) принимает вид:
\begin{equation*}
\begin{split}
h_{GZ(1\, 2\,3^3)}(s)=\sum_{j=0}^{3}\frac{s^{j+2}-1}{s-1}\Phi_{4-j}
		   =(s+1)\Phi_4&+(s^2+s+1)\Phi_3+(s^3+s^2+s+1)\Phi_2+\\&
		   +(s^4+s^3+s^2+s+1)\Phi_1.
\end{split}
\end{equation*}
Вычислим необходимые нам значения $\Phi_k$:
\begin{equation*}
\begin{split}
&\Phi_0=0,\ \Phi_1=1,\\
&\Phi_2=(s^2+s)\Phi_1-s^2\Phi_0=s^2+s,\\
&\Phi_3=(s^2+s)\Phi_2-s^2\Phi_1=s^4+2s^3,\\
&\Phi_4=(s^2+s)\Phi_3-s^2\Phi_2=s^6+3s^5+s^4-s^3.
\end{split}
\end{equation*}
Итак, 
\begin{equation*}
\begin{split}
h_{GZ(1\, 2\,3^3)}(s)&=(s+1)(s^6+3s^5+s^4-s^3)+(s^2+s+1)(s^4+2s^3)+\\
           &+(s^3+s^2+s+1)(s^2+s)+s^4+s^3+s^2+s+1=\\
           &=s^7+5s^6+8s^5+6s^4+4s^3+3s^2+2s+1
\end{split}
\end{equation*}
и соответствующий $h$-вектор $h_{GZ(1\, 2\,3^3)}=(1,2,3,4,6,8,5,1)$. При заданном значении $k$ формула (\ref{h(223^k)}) принимает вид:
\begin{equation*}
h_{GZ(2^2\,3^3)}(s)=\sum_{j=0}^{3}s^{j+2}\Phi_{3-j}+\frac{s^{4}-1}{s-1}=s^2\Phi_3+s^3\Phi_2+s^4\Phi_1+s^5\Phi_0+s^3+s^2+s+1.
\end{equation*}
Итак, \begin{equation*}
 \begin{split}
 h_{GZ(2^2\,3^3)}(s)&=s^2(s^4+2s^3)+s^3(s^2+s)+s^4+s^3+s^2+s+1=\\
                    &=s^4+2s^3+s^2+s+1
 \end{split}
 \end{equation*}
 и соответствующий $h$-вектор $h_{GZ(2^2\,3^3)}=(1,1,1,2,1)$.
\end{example}
\begin{corollary}\label{gfh123k223k}
А) Замкнутая формула для производящей функций последовательности $\big(h_{GZ(1\, 2\, 3^k)}(s)\big)_{k=0}^{\infty}$ имеет вид
\begin{equation}\label{gfh(123k)}
\begin{split}
H_{GZ(1\, 2\, 3^k)}(s,z)&=\sum_{k=0}^{\infty}h_{GZ(1\, 2\, 3^k)}(s) z^k=\frac{-sz+s+1}{(1-z)(1-sz)(s^2z^2-(s^2+s)z+1)};
\end{split}
\end{equation}
Б) Замкнутая формула для производящей функции последовательности
$\big(h_{GZ(2^2\,3^k)}(s)\big)_{k=0}^{\infty}$ имеет вид
\begin{equation}\label{gfh(223k)}
\begin{split}
H_{GZ(2^2\,3^k)}(s,z)&=\sum_{k=0}^{\infty}h_{GZ(2^2\,3^k)}(s) z^k=\frac{1}{(1-z)(s^2z^2-(s^2+s)z+1)}.
\end{split}
\end{equation}
\end{corollary}
\begin{proof}
Правая часть формулы (\ref{gfh(123k)}) получается раскрытием скобок и приведением подобных слагаемых в выражении
\begin{equation*}
\frac{1}{z}\bigg(\frac{1}{1-z}\cdot\frac{1}{1-sz}-1\bigg)\cdot \frac{\Phi(z)}{z},
\end{equation*}
  где функция $\Phi(z)$ была определена ранее формулой (\ref{Phi(z)}). Последующие выкладки доказывают утверждение А:
\begin{equation*}
\begin{split}
&\frac{1}{z}\bigg(\frac{1}{1-z}\cdot\frac{1}{1-sz}-1\bigg)\cdot \frac{\Phi(z)}{z}
=\frac{1}{z}\bigg(\sum_{j=0}^{\infty}\frac{s^{j+1}-1}{s-1}z^j-1\bigg)\cdot\frac{1}{z}\bigg(\sum_{k=0}^{\infty}\Phi_kz^k\bigg)=\\
&=\bigg(\sum_{j=0}^{\infty}\frac{s^{j+2}-1}{s-1}z^j\bigg)\cdot\bigg(\sum_{k=0}^{\infty}\Phi_{k+1}z^k\bigg)=
\sum_{k=0}^{\infty}\bigg(\sum_{j=0}^{k}\frac{s^{j+2}-1}{s-1}\cdot\Phi_{k-j+1}\bigg)z^k=\\
=&[\text{по формуле}(\ref{h(123^k)})]=\sum_{k=0}^{\infty}h_{GZ(1 2 3^k)}(s)z^k.
\end{split}
\end{equation*}
 Правая часть формулы (\ref{gfh(223k)}) получается раскрытием скобок и приведением подобных слагаемых в выражении

\begin{equation*}
\frac{s^2}{1-sz}\cdot \Phi(z)+\frac{1}{1-z}\cdot \frac{1}{1-sz},
\end{equation*}
где функция $\Phi(z)$ определена формулой (\ref{Phi(z)}). Последующие выкладки доказывают утверждение Б:
\begin{equation*}
\begin{split}
&\frac{s^2}{1-sz}\cdot \Phi(z)+\frac{1}{1-z}\cdot \frac{1}{1-sz}=\bigg(\sum_{j=0}^{\infty}s^{j+2}z^j\bigg)\cdot \bigg(\sum_{k=0}^{\infty}\Phi_k z^k \bigg)+\sum_{k=0}^{\infty}\frac{s^{k+1}+1}{s-1}z^k=\\
&=\sum_{k=0}^{\infty}\bigg(\sum_{j=0}^{k}s^{j+2}\cdot\Phi_{k-j}\bigg)z^k +\sum_{k=0}^{\infty}\frac{s^{k+1}-1}{s-1}z^k=[\text{по формуле}(\ref{h(223^k)})]=\sum_{k=0}^{\infty}h_{GZ(2^23^k)}(s).
\end{split}
\end{equation*}

\end{proof}
\section{Благодарности}Я благодарна моему научному руководителю В. А. Тиморину за постановку задачи и  внимание к работе на всех её этапах, М. Н. Вялому, Н. П. Долбилину и С. П. Тарасову за полезные обсуждения и замечания. Я признательна неизвестному рецензенту за внимательное изучение рукописи, исправление неточностей и ценные комментарии. Также я очень благодарна моему мужу Сереже Мелихову за поддержку и веру в меня.

\end{document}